\documentclass[12pt]{amsart}
\usepackage{amsmath, amssymb, amscd}
\usepackage[top=35truemm, bottom=35truemm, left=30truemm, right=30truemm]{geometry}

\newcommand{\Ubox}{\overline{\dim}_{\mathrm{B}}\:}

\renewcommand{\epsilon}{\varepsilon}
\renewcommand{\limsup}{\varlimsup}

\newtheorem{thm}{Theorem}

\newtheorem{lma}[thm]{Lemma}

\theoremstyle{definition}

\newtheorem*{rem*}{Remark}
\newtheorem*{note*}{Notation}

\begin{document}

\title[A fractal proof of the infinitude of primes]{A fractal proof of the infinitude of primes}

\author[K. Saito]{ Kota Saito }
\address{Kota Saito\\
Graduate School of Mathematics\\ Nagoya University\\ Furocho\\ Chikusa-ku\\ Nagoya\\ 464-8602\\ Japan }
\curraddr{}
\email{m17013b@math.nagoya-u.ac.jp}

\thanks{This paper will appear in \textit{Lithuanian Mathematical Journal}.}


\maketitle

\begin{abstract}
 This short paper gives another proof of the infinitude of primes by using the upper box dimension, which is one of fractal dimensions. 
\end{abstract}
A prime number is a natural number greater than 1 whose divisors are only 1 and itself. These numbers have been fascinating the whole human race. Euclid gave the first result on prime numbers in around 300 B.C. He showed the infinitude of primes: 
\begin{thm}[Euclid]\label{main1}
There are infinitely many prime numbers.
\end{thm}
His proof is arithmetical and simple (see \cite[Book IX Proposition~20]{Heath}). Nowadays, one can find numerous different proofs of the infinitude of primes. Surprisingly, the number of them is at least 183 from \cite{Mestrovic}. This paper also gives another proof of the infinitude of primes. We mainly use properties of the upper box dimension and the fact that any natural number greater than 1 can be written as a product of prime numbers (the uniqueness of the factorization is not required). Our method is close to Euler's idea on the divergence of the sum of reciprocals of prime numbers \cite[pp.~172-174]{Euler} (alternatively see \cite[pp.~1-2]{Davenport}). From his proof, $\sum_{p\,:\,\text{prime}} 1/p=\infty$ is equivalent to $\sum_{n\in \mathbb{N}}1/n=\infty$, which implies that the size of the set of all positive integers should be small if there were only finitely many prime numbers. Similarly, we will find that the upper box dimension of the set of reciprocals of all positive integers should be small (zero) if there were only finitely many prime numbers. This leads to a contradiction.\\

Here we define {\it the upper box dimension}. Let $F \subset \mathbb{R}$ be a bounded set, and $\delta$ be a positive number. A family of sets $\{U_j\}_{j=1}^N$ is called a $\delta$-cover of $F$ if $F\subseteq \bigcup_{j=1}^N U_j$ and $d(U_j)\leq \delta$ for all $j=1,2,\ldots,N$, where $d(U_j)$ denotes the diameter of $U_j$ $i.e.$ we define $d(U)= \sup_{x,y\in U} |x-y|$ for every bounded set $U$. Then we define {\it the upper box dimension of} $F$ as
\[
	 \Ubox F= \limsup_{\delta\rightarrow+0}  \frac{\log N(F,\delta)}{-\log\delta},
\]
where $N(F,\delta)$ denotes the smallest cardinality of a $\delta$-cover of $F$ {\it i.e.} 
\[
	N(F,\delta)=\min\left\{N\in \mathbb{N}\: \colon\: \{U_{j}\}_{j=1}^N \text{ is a $\delta$-cover of $F$} \right\}.
\]
We refer \cite{Falconer,Robinson} to the readers who are interested in more details.

In order to prove Theorem~\ref{main1}, we show the following lemma:
\begin{lma}\label{main2}
Let $C,D\subset \mathbb{R}$ be bounded sets. We have
\[
    \Ubox (C D) \leq \Ubox C + \Ubox D,  
\]
where $CD=\{cd\in \mathbb{R}\: :\: c\in C,\ d\in D \}$. 
\end{lma}
By induction and Lemma~\ref{main2}, we immediately obtain that
\begin{equation}\label{iterate}
    \Ubox (C_1\cdots C_n) \leq \Ubox C_1 +\cdots + \Ubox C_n,  
\end{equation}
for all bounded sets $C_1,\ldots, C_n\subset \mathbb{R}$.
\begin{proof}[Proof of Theorem~\ref{main1} assuming Lemma~\ref{main2}]
Let $1/\mathbb{N}=\{1/n\: :\: n\in\mathbb{N} \}$. From \cite[Example~3.5]{Falconer}, we find that 
$
	\Ubox (1/\mathbb{N})=1/2.
$
Here we need just one direction of this formula, that is
\begin{equation}\label{dim1/N}
\Ubox (1/\mathbb{N})\geq 1/2.
\end{equation}
We show this inequality. Fix any $0<\delta<1/2$ and let $k\geq 2$ be the integer such that 
\[
1/(k(k+1)) \leq \delta < 1/((k-1)k).
\]
If $U\subset \mathbb{R}$ with $d(U)\leq \delta$, then $U$ has at most one intersection with the set $\{1,1/2,\ldots, 1/k\}$ since if $1/s,\:1/t\in U$ holds for some $1\leq s<t\leq k$, then we have
\[
	d(U)\geq \frac{1}{s}-\frac{1}{t}\geq \frac{1}{st}\geq \frac{1}{k(k-1)} >\delta, 
\]
 which is a contradiction. Therefore if we take any $\delta$-cover $\{U_j\}_{j=1}^N$ of $1/\mathbb{N}$, then $N\geq k$, which implies that $N(1/\mathbb{N},\delta)\geq k \geq (k(k+1))^{1/2}/2\geq \delta^{-1/2}/2$. Therefore we have
 \[
	 \Ubox (1/\mathbb{N})= \limsup_{\delta\rightarrow+0}  \frac{\log N(1/\mathbb{N},\delta)}{-\log\delta}
	 \geq \limsup_{\delta\rightarrow+0}  \frac{\log(\delta^{-1/2}/2)}{-\log\delta}=1/2.
\]

Let $A(p)=\{1/p^k \: \colon\: k=0,1,2,\ldots\}$ for every prime number $p$. Then $\Ubox A(p)=0$ holds. In fact, for every $0<\delta<1/2$, it follows that
\[
A(p)\subseteq [-\delta/2,\delta/2] \cup \bigcup_{0\leq k\leq\frac{\log (2/\delta)}{\log p}}  [-\delta/2+1/p^k,\delta/2+1/p^k].  
\]
This yields that
$
 N\big(A(p),\delta\big) \leq 2+\frac{\log(2/\delta)}{\log p}.
$
Hence we have
\begin{equation}\label{dimA(p)}
0\leq \Ubox A(p) = \limsup_{\delta\rightarrow +0} \frac{\log N\big(A(p),\delta\big)}{-\log \delta}\leq \limsup_{\delta\rightarrow +0} \frac{\log\left( 2+\frac{\log(2/\delta)}{\log p}\right) }{\log (1/\delta)}=0.
\end{equation}
Assume that there are only finitely many prime numbers $p_1,\ldots,p_n$. Then we have $1/\mathbb{N}= A(p_1)\cdots A(p_n)$
since any natural number greater than 1 can be written as a product of prime numbers. Therefore we obtain
\[
 1/2\leq \Ubox 1/\mathbb{N}= \Ubox (A(p_1)\cdots A(p_n))\leq  \Ubox A(p_1)+\cdots +\Ubox A(p_n)=0
\]
by  (\ref{iterate}), (\ref{dim1/N}), and (\ref{dimA(p)}). This is a contradiction.
\end{proof}

\begin{proof}[Proof of Lemma~\ref{main2}]

Fix any $0<\delta<1/2$. Let $R$ be a sufficiently large number such that $C\subseteq [-R,R]$ and $D\subseteq [-R,R]$. Let $\{U_i\}_{i=1}^{N_C}$ and $\{V_j\}_{j=1}^{N_D}$ be $\delta/(2R)$-covers of $C$ and $D$, respectively, where we define $N_C=N(C,\delta/(2R))$ and $N_D=N(D,\delta/(2R))$. Then we find that $\{U_i V_j\}_{1\leq i \leq N_C, 1\leq j\leq N_D}$ is a $\delta$-cover of $CD$. Indeed, for fixed $1\leq i \leq N_C$, $1\leq j\leq N_D$, and for all $c_1,c_2\in U_i$ and $d_1,d_2\in V_j$ we have
\begin{equation}\label{Lipschitz}
 |c_1\cdot d_1-c_2\cdot d_2 |\leq |c_1||d_1-d_2|+|d_2||c_1-c_2|\leq R\cdot \delta/(2R) + R\cdot \delta/(2R)=\delta, 
\end{equation}
which means that the diameter of $U_i V_j$ is at most $\delta$. Furthermore, it is clear that 
\[
	C D\subseteq \bigcup_{1\leq i \leq N_C, 1\leq j\leq N_D} U_i V_j. 
\]
Therefore the following inequality holds:
\[
N(C D, \delta)\leq N_C\cdot N_D=N(C,\delta/(2R))\cdot N(D,\delta/(2R)),
\]
which yields that
\begin{align*}
\Ubox (C D) &= \limsup_{\delta\rightarrow +0}\frac{\log N(CD,\delta)}{-\log \delta} \\
&\leq \limsup_{\delta\rightarrow +0}\frac{\log N(C,\delta/(2R))}{-\log (\delta/2R)-\log(2R)}
+\limsup_{\delta\rightarrow +0}\frac{\log N(D,\delta/(2R))}{-\log (\delta/2R)-\log(2R)}\\
&=\Ubox C+\Ubox D.
\end{align*}
\end{proof}

The above proof does not require any specific knowledge on fractal geometry, but we can get a much simpler proof if we use properties of the upper box dimension.
\begin{proof}[Another proof of Lemma~\ref{main2}]
We define $\phi\: :\: C\times D \rightarrow  \mathbb{R}$ as $\phi(c,d)=cd$, where $C\times D=\{(c,d)\in\mathbb{R}^2\ \colon\ c\in C, d\in D \}$. We show that $\phi$ is Lipschitz-continuous. Let $R$ be a sufficiently large number such that $C\subseteq [-R,R]$ and $D\subseteq [-R,R]$. Then we have
\[
|\phi(c_1,d_1)-\phi(c_2,d_2)|= |c_1||d_1-d_2|+|d_2||c_1-c_2|\leq 2R\|(c_1,d_1)-(c_2,d_2)\|
\]
for all $c_1,c_2\in C$ and $d_1,d_2\in D$, where $\|\cdot\|$ denotes the Euclidean norm on $\mathbb{R}^2$. Thus $\phi$ is Lipschitz-continuous. Therefore we have 
\[
\Ubox (CD)=\Ubox \phi(C\times D) \leq \Ubox (C\times D) \leq \Ubox C +\Ubox D
\]
by \cite[Proposition~2.5 and Product formula~7.5]{Falconer}.

\end{proof}

\section*{Acknowledgement}
The author would like to thank my supervisor Professor Kohji Mastumoto, my previous supervisor Professor Neal Bez, and the referee for useful comments. The author is grateful to the seminar members for discussion. The author is financially supported by Yoshida Scholarship Foundation.

\end{document}